\RequirePackage{amsmath}
\documentclass[12pt]{llncs}
\usepackage[colorlinks=true, allcolors=blue,linkcolor=blue,citecolor=blue,backref=page,bookmarksopen=true]{hyperref}
\usepackage{bookmark}
\usepackage{graphicx}
\usepackage[dvipsnames]{xcolor}
\usepackage{listings}
\usepackage{amssymb}
\usepackage{amsfonts}
\usepackage{mathtools}
\usepackage{maple}
\usepackage[utf8]{inputenc}
\usepackage{breqn}
\usepackage{orcidlink}
\usepackage{mathrsfs}
\usepackage{listings}
\usepackage{microtype}
\usepackage{moreverb}
\usepackage{verbatim}
\usepackage{sverb}
\usepackage{geometry}
\usepackage{color}
\usepackage{makecell}
\usepackage{url}
\usepackage{cleveref}
\usepackage{moreverb}
\usepackage{verbatim}
\usepackage{sverb}
\usepackage{algorithm}
\usepackage{algpseudocode}
\usepackage{algorithmicx}
\usepackage{euscript}
\usepackage{abstract}
\usepackage[T1]{fontenc}
\usepackage{lmodern}
\usepackage{textcomp}
\usepackage{stmaryrd}
\SetSymbolFont{stmry}{bold}{U}{stmry}{m}{n}

\newcommand*{\CQFD}{\hfill\ensuremath{\square}}

\DeclareMathOperator{\ssm}{Sm}
\DeclareMathOperator{\ssmr}{Smr}
\DeclareMathOperator{\ssmx}{Smp}
\newcommand{\NN}{\mathbb{N}}

\newcommand{\ZZ}{\mathbb{Z}}
\newcommand{\QQ}{\mathbb{Q}}
\newcommand{\KK}{\mathbb{K}}

\newcommand{\g}{\mbox{\tiny $g$}}
\newcommand{\sm}{\mbox{\small $\ssm$}}
\newcommand{\smr}{\mbox{\small $\ssmr$}}
\newcommand{\smx}{\mbox{\small $\ssmx$}}
\newcommand{\sx}{\mbox{$s_*$}}
\newcommand{\rs}{\mbox{$R_{\sigma}$}}
\newcommand{\mmod}[1]{\ \mathrm{mod}\ #1}
\DeclareMathOperator{\ord}{ord}
\DeclareMathOperator{\sol}{Sol}

\oddsidemargin \evensidemargin
\begin{document}

\title{Concatenations of Terms of an Arithmetic Progression}

\author{Florian Luca\orcidlink{0000-0003-1321-4422}\inst{1} \and Bertrand Teguia Tabuguia\orcidlink{0000-0001-9199-7077}\inst{2}}

\institute{School of Maths, Wits University, Johannesburg, South Africa\\
    \email{florian.luca@wits.ac.za}
    \and 
Department of Computer Science, University of Oxford, UK\\
    \email{bertrand.teguia@cs.ox.ac.uk}}

\maketitle              

\begin{abstract}

\vspace*{-.5em}

Let $\left(u(n)\right)_{n\in\NN}$ be an arithmetic progression of natural integers in base $b\in\NN\setminus \{0,1\}$. We consider the following sequences: {\small $s(n)=\overline{u(0)u(1)\cdots u(n) }^b$} formed by concatenating the first $n+1$ terms of $\left(u(n)\right)_{n\in\NN}$ in base $b$ from the right; {\small $s_{\g}(n) = \overline{u(n)u(n-1)\cdots u(0)}^b$}; and $\left(s_*(n)\right)_{n\in\NN}$, given by {\small $s_*(0)=u(0)$, $s_*(n)=\overline{s(n)s_{\g}(n-1)}^b, n\geq 1$}. We construct explicit formulae for these sequences and use basic concepts of linear difference operators to prove they are not P-recursive (holonomic). We also present an alternative proof that follows directly from their definitions. We implemented $\left(s(n)\right)_{n\in\NN}$ and $\left(s_{\g}(n)\right)_{n\in\NN}$ in the decimal base when $(u(n))_{n\in\NN}=\NN\setminus \{0\}$.

\end{abstract}

\section{Introduction}
    Consider an integer sequence $(u(n))_{n\in\NN}$ (or simply $(u(n))_n$) and let $i,j\in\NN$. The integers $\overline{u(i)u(j)}$ and $\overline{u(j)u(i)}$ obtained by ``gluing together'' $u(i)$ and $u(j)$ defines two distinct {\it concatenations}. For example, if $(u(n))_n=\NN\setminus \{0\}=\{1,2,3,\ldots\}$, then $m_1=\overline{u(2)u(9)}=310$ and $m_2=\overline{u(9)u(2)}=103$. For concatenations of concatenations we use one overline instead of three. For instance, $310103=\overline{m_1m_2}$ is simply $\overline{u(2)u(9)u(9)u(2)}$ instead of $\overline{\overline{u(2)u(9)}\,\overline{u(9)u(2)}}$. When the digits involved in a concatenation are from a number base $b$, we say {\it concatenation in base $b$}. The above $m_1$ and $m_2$ may be seen as concatenations in a base $b\geq 4$. In general we write $\overline{u(i)u(j)}^b$ to specify the base in which we concatenate.
    
    In this article, we study concatenations of the sequence $\left(u(n)\right)_{n\in\NN}$ $ \coloneqq \left(u(0)+d\,n\right)_{n\in\NN}$, with $d\in\NN\setminus \{0\}$, in an arbitrary number base $b\in\NN, b\geq 2$. The most natural of such sequences is the (ordered) set of positive integers $\NN\setminus \{0\}$ in the decimal base with $u(0)=d=1$. For that sequence, one considers $\left(\sm(n)\right)_{n\in\NN} \coloneqq \left(\overline{12\cdots(n+1)}^{10}\right)_n$, $\left(\smr(n)\right)_{n\in\NN}=\left(\overline{(n+1)n\cdots1}^{10}\right)_n$, and $\left(\smx(n)\right)_{n\in\NN}$, $\smx(0)=1, \smx(n)=\overline{\sm(n)\smr(n-1)}^{10}, n\geq 1,$ where $\overline{(.)}^b$ denotes the concatenation in base $b$. Some authors \cite{SmF2,SmF1,torres2004smarandache} attribute $\left(\sm(n)\right)_{n\in\NN}$ to Smarandache. We adopt this appellation to single out the particular case of positive integers in the decimal base. It is worth mentioning the connection with the Champernowne constant $0.\overline{123\ldots891011\ldots}^{10}$ whose integer in its fractional part is the limit of $(\sm(n))_n$, also known as the Champernowne word. One motivation of this paper is that no term $\sm(n)$ is known to be prime for $n\leq 10^6$ \cite{mers1}. There is an ongoing Sieve computation on the mersenneforum.org at \cite{mers2} for $n\leq 10^{15}$. For more details about this sequence, see \href{https://oeis.org/A007908}{A007908} from \cite{Oeis}.

    Using Pad\'e approximants \cite{gfun,yurkevich2022art} with the $9$ first coefficients of the generating series
    \begin{equation}\label{eq:pade1}
        F(x)\coloneqq \sum_{n=0}^{\infty}\sm(n)\,x^n = 1 + 12\,x+\cdots+1234567\,x^8 + O(x^9),
    \end{equation}
   one finds the rational function
   \begin{equation}\label{eq:form1}
       \frac{1}{(10\,x-1)(x-1)^2} = \moverset{\infty}{\munderset{n =0}{\textcolor{gray}{\sum}}}\! \left(\frac{100 \,10^{n}}{81}-\frac{n}{9}-\frac{19}{81}\right) x^{n},
   \end{equation}
    which approximates $F(x)$ accurately up to order $9$, covering thus all $1$-digit concatenations in $\left(\sm(n)\right)_{n\in\NN}$. The last equality in \eqref{eq:form1} is obtained automatically with the algorithms from \cite{BTphd,BTmc2020,teguia2021symbolic}. However, in this case, one can perform these computations by hand or with classical methods. We investigate the intermediate steps of that algorithm to derive formulae for arbitrary concatenations of $\left(u(n)\right)_n$ in base $b$. 

    Throughout this paper, {\it right-concatenation} of $(u(n))_n$ refers to the sequence of general term $s(n)=\overline{u(0)u(1)\cdots u(n) }^b$, formed by concatenating the first $n+1$ terms of $\left(u(n)\right)_n$ in base $b$ from the right. Similarly, {\it left-concatenation} of $(u(n))_n$ is associated to the sequence of general term $s_{\g}(n) = \overline{u(n)u(n-1)\cdots u(0)}^b$ ({\tiny $g$} for left in French). {\it The palindromic} concatenation of $(u(n))_n$ defines $\left(s_*(n)\right)_{n\in\NN}$, given by $s_*(0)=u(0)$, $s_*(n)=\overline{s(n)s_{\g}(n-1)}^b, n\geq 1$. Depending on the context, `right,' `left,' and `palindromic' may be omitted. In this paper, we give precise answers to the following questions:
 	\begin{enumerate}
 		\item Given a positive integer $l$, what are the recurrence equations for $l$-digit concatenations?
 		\item How do the solutions of these equations relate to $l$-digit concatenations?
 		\item Does this lead to general formulae for $s(n),s_{\g}(n),$ and $\sx(n)$, $n\in\NN$?
            \item Do these concatenations obey a linear recurrence with polynomial coefficients in $n$? 
 	\end{enumerate}
 	
     Note that except for the initial version of this paper \cite{teguia2022explicit}, and the independent work from \cite[Section 4]{alekseyev2020three}, which deals with right-concatenations, we are not aware of any other work that studies these sequences. Other references, such as \cite{torres2004smarandache}, essentially explore these sequences with naive algorithms.
 	
 	In \Cref{sec2}, we establish recurrence equations encoding fixed-length concatenations. In \Cref{sec3}, we solve those recurrence equations and deduce formulae associated with fixed-length concatenations. We will then deduce explicit formulae to compute $s(n), s_{\g}(n),$ and $s_*(n)$. \Cref{sec4} is devoted to the proof that our concatenating sequences do not obey linear recurrence equations with polynomial coefficients. Our first proof shows this as a general fact about sequences with linearly independent holonomic representations in infinitely many integer intervals. Our second proof arises from divisibility criteria deduced from asymptotic terms of these sequences. \Cref{sec5} presents an algorithm to compute $\sm(n)$, from which the one of $\smr(n)$ can be derived easily. We also present some computations with our implementation in the Maple computer algebra system \cite{maple}.

	\section{Recurrence equations}\label{sec2}
 
	A holonomic recurrence equation (RE) is a linear homogeneous recurrence equation with polynomial coefficients in the index variable. Although all the recurrence equations in this section have constant coefficients (C-finite), in \Cref{sec4}, we will see that only one is of minimal order. The other two equations can be reduced to second-order holonomic REs with polynomial coefficients of degree $1$. These observations will serve as a premise for our forthcoming proof in \Cref{sec4}, which shows that `{\it different holonomicity on an infinite sequence of range implies non-holonomicity}'. Also, the solver used in \Cref{sec3} effectively applies to holonomic recurrence equations. As observed in the introduction, from the use of the symbolic-computation algorithm from \cite{teguia2021symbolic}, the formula to compute Smarandache numbers for $1$-digit right-concatenations is obtained from solving a holonomic RE. 
 
    In this section, we use the indeterminate term $a(n)$ for arbitrary holonomic REs. To start, let us proceed as the guess-and-prove paradigm (see \cite{polya2004solve}) would suggest by observing (guessing) $l$-digit right-concatenations of natural numbers for small $l$; that is, observing recurrence relations in $(\sm(n))_n$ for $l=1,2,3$. Using the \texttt{GFUN} package \cite{gfun} (procedure \texttt{listtorec}) with enough initial coefficients, we observe the following:
	\begin{enumerate}
		\item A holonomic RE for $1$-digit right-concatenations is given by
		\begin{equation}\label{9}
			a(n+3) - 12 \cdot a(n+2) + 21 \cdot a(n+1) - 10 \cdot a(n) =0.
		\end{equation}
		
		\item For $2$-digit right-concatenations we get
		\begin{equation}\label{10}
			a(n+3) - 102 \cdot a(n+2) + 201 \cdot a(n+1) - 100 \cdot a(n) =0.
		\end{equation}
		
		\item For $3$-digit right-concatenations we have
		\begin{equation}\label{11}
			a(n+3) - 1002 \cdot a(n+2) + 2001 \cdot a(n+1) - 1000 \cdot a(n) =0.
		\end{equation}
	\end{enumerate}

	In general, following this observation, we conjecture and prove that $l$-digit concatenations satisfy linear recurrences with constant coefficients that depend on $b$ and $l$.
	
	\begin{lemma}\label{lem1} Let $l$ be a positive integer. The terms of $l$-digit concatenations in $(\sm(n))_n$ satisfy the recurrence equation
		\begin{equation}\label{12}
			a(n+3) - (10^l+2) \cdot a(n+2) + (2\cdot 10^l + 1) \cdot a(n+1) - 10^l \cdot a(n) =0.
		\end{equation}
	\end{lemma}
	\begin{proof}
		Assume $10^{l-1}-1\leqslant n \leqslant 10^{l}-2$. Let $N$ be the number occupying the last $l$ digits of $\sm(n+1)$. We have the relationships
		\begin{eqnarray}
			\sm(n+1) &=& 10^l \cdot \sm(n) + N \label{13},\\
			\sm(n+2) &=& 10^l \cdot \sm(n+1) + N+1 \label{14},\\
			\sm(n+3) &=& 10^l \cdot \sm(n+2) + N+2 \label{15}.
		\end{eqnarray}
	 From \eqref{13} we get $N=\sm(n+1) - 10^l \sm(n)$, and from \eqref{14}, $1=\sm(n+2)-10^l \sm(n+1) - N$. Therefore, we can write
	 \begin{equation}\label{16}
	 	2 \times 1 = 2 \cdot \left(\sm(n+2) - (10^l+1) \cdot \sm(n+1) + 10^l \cdot \sm(n)\right).
	 \end{equation}
 	Finally, we substitute $2$ in \eqref{15} by the right-hand side of \eqref{16} and obtain that
 	\begin{equation}\label{17}
 		\sm(n+3) = (10^l+2) \cdot \sm(n+2) - (2\cdot 10^l + 1) \cdot \sm(n+1) + 10^l \cdot \sm(n),
 	\end{equation}
 	which shows that terms of $l$-digit concatenations in $(\sm(n))_n$ satisfy \eqref{12}. \CQFD
	\end{proof}

	The above proof can be used as a template to prove that all terms corresponding to a fixed-length concatenation in  $(s(n))_{n}$ satisfy a C-finite recurrence equation whose coefficients depend on the length of the concatenation. The following lemma gives the recurrence equations for right-concatenations and left-concatenations of an arithmetic progression. We ignore the case where the common difference disqualifies $l$-digit concatenations. E.g., in the decimal base, there is no $1$-digit concatenation with $10$ as the common difference. We assume that the $l$-digit concatenations involve at least $4$ terms in $(s(n))_n$, c.f. remark \ref{rmk1}.
	
	\begin{lemma}\label{lem2} Let $l$ be a positive integer and $b\geq 5$ a natural number base. 
	\begin{itemize}
		\item[i.] The terms of $l$-digit concatenations in $(s(n))_n$ satisfy the recurrence equation
        \begin{equation}\label{eq:eq12}
            a(n+3) - (b^l+2) \cdot a(n+2) + (2\cdot b^l + 1) \cdot a(n+1) - b^l \cdot a(n) =0.
        \end{equation}
		\item[ii.] The recurrence equation for $l$-digit concatenations in $(s_{\g}(n))_n)$ is
		\begin{equation}\label{18}
			a(n+3) - (2 \cdot b^l + 1) \cdot a(n+2) + (b^{2l}+2\cdot b^l) \cdot a(n+1) - b^{2l} \cdot a(n) = 0.
		\end{equation}
	\end{itemize}
	\end{lemma}
 
  \begin{remark}\label{rmk1}
     We note that when $b\leq 4$, for the above recurrences to hold, we need $l\geq 3$ if $b=2$, and $l\geq 2$ if $b\in\{3,4\}$.
 \end{remark}
 
	\begin{proof} 
	For \textit{i}, it suffices to observe that \eqref{14} becomes $s(n+2)=b^l s(n+1) + N + d$, where $d$ is the common difference of $(u(n))_n$. The proof is then straightforward by substituting $d$ and $2 d$ similarly as we did for $1$ and $2$ in the proof of \Cref{lem1}.
		
	For \textit{ii}, we know that $s_{\g}(0)=u(0)$. Suppose $b^{l-1}-u(0)\leq n\cdot d \leq b^{l}-u(0)$. Let $N$ be the first $l$ digit of $s_{\g}(n+1)$, and $\nu$ the digit length of $s_{\g}(n)$. We have the relationships
	\begin{eqnarray}
		s_{\g}(n+1) &=&  N \cdot b^{\nu} + s_{\g}(n) \label{19},\\
		s_{\g}(n+2) &=& (N+d) \cdot b^{l + \nu} + s_{\g}(n+1) = b^l \left(N\,b^{\nu}\right) + d\,b^{l +\nu} + s_{\g}(n+1) \label{20},\\
		s_{\g}(n+3) &=& (N+2 d) \cdot b^{2\,l + \nu} + s_{\g}(n+2) \label{21}.
	\end{eqnarray}
	From \eqref{19}, we get $N\,b^{\nu} = s_{\g}(n+1)-s_{\g}(n)$, and from \eqref{20}, $d\, b^{l + \nu} = s_{\g}(n+2)-b^l \cdot (N\,b^{\nu})-s_{\g}(n+1)$. Therefore
	\begin{equation}\label{22}
	2 \cdot d  \cdot b^{2\,l+\nu} = 2 \cdot b^l \left(s_{\g}(n+2)-(b^l+1) \cdot s_{\g}(n+1)+b^l \cdot s_{\g}(n)\right).
	\end{equation}
	Finally, by substitution in \eqref{21} we get
	\begin{eqnarray}
		s_{\g}(n+3) &=& b^{2l} (s_{\g}(n+1)-s_{\g}(n))\nonumber\\
		       &\phantom{=}& + 2 \cdot b^l\left(s_{\g}(n+2)-(b^l+1) s_{\g}(n+1)+b^l s_{\g}(n)\right) + s_{\g}(n+2) \nonumber\\
			   &=& (2 \cdot b^l + 1) \cdot s_{\g}(n+2) - (b^{2l}+2 \cdot b^l) \cdot s_{\g}(n+1) + b^{2l} \cdot s_{\g}(n) \label{23},
	\end{eqnarray}
	which concludes the proof. \CQFD
	\end{proof}

	We can similarly define recurrence equations for the palindromic concatenation, which corresponds to $(\sx(n))_n$. A typical example in the decimal case is $(\smx(n))_n$, for which the $10$th term is prime.
	
	\begin{lemma}\label{lem3} Let $l$ and $b\geq 5$ be positive integers. The terms of $l$-digit concatenations in $(\sx(n))_n$ satisfy the recurrence equation
		\begin{equation}\label{24}
			a(n+3) - \left(1+b^l+b^{2l}\right) \cdot a(n+2) +  \left(b^l+b^{2l}+b^{3l}\right)\cdot a(n+1) - b^{3l} \cdot a(n) = 0.
		\end{equation}
	\end{lemma}
	A technique to prove Lemma \ref{lem3} easily followed from the proofs of Lemma \ref{lem2}.
	
	The recurrence equations in \Cref{lem2} and \Cref{lem3} encode all fixed-length concatenations of our sequences $(s(n))_n,$ $(s_{\g}(n))_n$, and $(\sx(n))_n$. Of course, several other concatenations can similarly be considered. In particular, one can also construct recurrence equations for sequences like $\left(\overline{123\ldots n(n+1)(n+1)n\ldots321}^{10}\right)_n$, $n\geq0$: $11$, $1221$, $123321$, $\ldots$, for which the $10$th term is also a prime number. We understand that all these other cases can be addressed with a similar reasoning. In this paper, we only concentrate on the concatenations of \Cref{lem2} and \Cref{lem3}.
	
	\section{Formulae for concatenations}\label{sec3}
	
	All recurrence equations of the previous section are third-order linear recurrences. Therefore, each of them has three linearly independent solutions. Although all these recurrences might be solve using classical techniques for linear recurrence with constant coefficients, we solve them by using the algorithm in \cite{BThyper}, which implements a variant of van Hoeij's algorithm (see \cite{petkovvsek1992hypergeometric,van1999finite}). This choice is due to the presence of the parameters $b,l,$ in the equations which are unlikely to be considered with existing constant-coefficient linear recurrence solvers. The formulae of our concatenating sequences are deduced as linear combinations of terms in the bases of computed solutions.
	
	As usual, we start with the $l$-digit right-concatenation which includes $(\sm(n))_{n}$. For $n_1,n_2\in~\NN, n_1\leq n_2$, we denote by $\llbracket n_1, n_2\rrbracket$ the set of integers $\left\lbrace n_1,n_1+1,\ldots,n_2\right\rbrace$.
	
	\begin{theorem}\label{theo1} The general term of $(s(n))_n,\, b\geq 5,$ can be computed as follows: 
    \begin{align}
			&s(n) = \alpha_l + \mu_l (n-t_l) + \theta_l\,b^{l(n-t_l)}, \label{25}\\
			&l = \lceil \log_{b}(n\,d + s(0) + 1)\rceil,~ t_l = \left\lceil\frac{b^{l-1}-s(0)}{d}\right\rceil, \label{26}\\
			&\alpha_l = - \frac{\left(b^l-1\right)\cdot u(t_l) + d\cdot b^l}{\left(b^l-1\right)^2}, \label{27}\\
			&\mu_l =-\frac{d}{b^l-1},\label{28}\\
			&\theta_l = \frac{\kappa_2-2 \cdot \kappa_1 + \kappa_0}{\left(b^l-1\right)^2}, \label{29}\\
			&\kappa_0=s(t_l),~ \kappa_1=s(t_l+1),~ \kappa_2=s(t_l+2).\label{30}
	\end{align}
	\end{theorem}
	\begin{proof}
		The recurrence equation for $l$-digit right-concatenation is (see Lemma \ref{lem2})
		\begin{equation*}
			a(n+3) - (2\cdot b^l + 1)\cdot a(n+2) + (b^{2l}+2\cdot b^l)\cdot a(n+1) - b^{2l}\cdot a(n) = 0.
		\end{equation*}
	    Using the algorithm in \cite{BThyper}, we find the basis of solutions
	    \begin{equation}\label{31}
	    	\left\{1,~ n,~ b^{l\,n} \right\},
	    \end{equation}
    	which can be easily verified. Therefore there exist constants $\alpha_l, \mu_l, \theta_l$ to compute terms of $l$-digit concatenations in $(s(n))_{n}$ as follows:
    	\begin{equation}\label{32}
    		 s(n)=\alpha_l + \mu_l\,n + \theta_l\,b^{ln}.
    	\end{equation}
    	The constants $\alpha_l, \mu_l, \theta_l$ can be computed by solving the linear system
    	\begin{equation}\label{33}
    		\begin{cases}
    			\alpha_l + \theta_l = \kappa_0\\
    			\alpha_l + \mu_l + \theta_l b^l = \kappa_1\\
    			\alpha_l + 2 \mu_l + \theta_l b^{2l} = \kappa_2
    		\end{cases},
    	\end{equation}
    	where $\kappa_0, \kappa_1,$ and $\kappa_2$ correspond to the first three $l$-digit concatenations in $(s(n))_{n}$; these are, respectively, $s(t_l)$, $s(t_l+1)$, and $s(t_l+2)$, where $t_l=\left\lceil\left(b^{l-1}-s(0)\right)/d\right\rceil$, and $l=\lceil \log_{b}(n\,d + s(0) + 1)\rceil$. Solving \eqref{33} yields 
    	\begin{equation}\label{34}
    		\alpha_l = \frac{2\cdot \left(\kappa_1-b^l\cdot \kappa_0\right) - \left(\kappa_2-b^{2l}\cdot \kappa_0\right)}{\left(b^l-1\right)^2},~ \mu_l =\frac{\kappa_2-b^l\cdot \kappa_1 - \left(\kappa_1 - b^l\cdot \kappa_0\right)}{b^l-1},
    	\end{equation}
    	and $\theta_l$ as in \eqref{29}. The coefficients $\alpha_l$ and $\mu_l$ can be further simplified by using properties of the arithmetic progression $(u(n))_{n}$. It is easy to see that
    \begin{align}
    &s(t_l+1) - b^l\cdot s(t_l) = u(t_l) + d, \label{35}\\
    &s(t_l+2) - b^{2l}\cdot s(t_l) = \left(u(t_l)+d\right)\cdot b^l + u(t_l)+2\cdot d,\label{36}\\
    &s(t_l+2) - b^l\cdot s(t_l+1) = u(t_l)+2\cdot d. \label{37}
    \end{align}
    	After substitution in \eqref{34} we find $\alpha_l$ and $\theta_l$ as expected.
    	
    	Finally, to use \eqref{32} as the formula to compute $s(n)$ for all non-negative integers $n$, we shift the index variable $n$ in the range $\llbracket 0, t_{l+1}-t_l-1\rrbracket$ by substituting $n$ by $n-t_l$. \CQFD
	\end{proof}
	
	From \Cref{theo1}, we see that an efficient computation of $l$-digit concatenations need only compute $\alpha_l$, $\mu_l$, and $\theta_l$ once. For example, for $(\sm(n))_{n}$, this yields effective formulae for indices in ranges like $\llbracket 10^6-1, 10^7-2\rrbracket$, $\llbracket 10^7-1, 10^8-2\rrbracket$, etc.
	
	We mention that $\theta_l$ in \Cref{theo1} can be written in terms of $\kappa_0$, $u(t_l)$, and $d$. Nevertheless, it seems more efficient to compute $\theta_l$ with $\kappa_0$, $\kappa_1$ and $\kappa_2$, which are respectively deduced by the right-concatenation of $u(t_l)$ to $s(t_l-1)$, $u(t_l)+d$ to $\kappa_0$, and $u(t_l)+2d$ to $\kappa_1$. The formula for Smarandache numbers is a direct consequence of \Cref{theo1}.

	\begin{corollary}[Formula for Smarandache numbers]\label{cor1} The general term of $(\sm(n))_n$ can be computed as follows:
 \begin{equation}
     \sm(n) = \alpha_l + \mu_l (n-t_l) + \theta_l 10^{l(n-t_l)}, \label{38}
 \end{equation}
where 
\[l = \lceil \log_{10}\left(n + 2\right)\rceil,~ t_l = 10^{l-1}-1, \alpha_l = -\frac{10^{2 l - 1} + 9 \cdot 10^{l-1}}{(10^l-1)^2},\]
\[\mu_l = -\frac{1}{10^l-1}, \theta_l = \frac{\kappa_2-2\cdot \kappa_1 + \kappa_0}{(10^l-1)^2}, \kappa_0=\sm(t_l),~ \kappa_1=\sm(t_l+1),~ \kappa_2=\sm(t_l+2).\]
\end{corollary}

Let us now give the formula for left-concatenations encoded by $(s_{\g}(n))_n$.
	
\begin{theorem}\label{theo2} The general term of $(s_{\g}(n))_n,\, b\geq5$, can be computed as follows:
\begin{align}
    &s_{\g}(n) = \alpha_l + \mu_l \cdot b^{l(n-t_l)} + \theta_l\cdot (n-t_l)\cdot b^{l(n-t_l)} \label{44}\\
    &l = \lceil \log_{b}(n\,d + s_{\g}(0) + 1)\rceil,~ t_l = \left\lceil\frac{b^{l-1}-s_{\g}(0)}{d}\right\rceil,~ \nu_l\equiv \textnormal{ digit length of } s_{\g}(t_l), \label{45}\\
    &\alpha_l = \frac{\kappa_2-2\cdot b^l\cdot \kappa_1 + b^{2l}\cdot \kappa_0}{(b^l-1)^2},\, \mu_l =\frac{\left(\left(b^l-1\right)\cdot u(t_l) - d\right)\cdot b^{\nu_l}}{\left(b^l-1\right)^2},\, \theta_l = \frac{d\cdot b^{\nu_l}}{b^l-1}, \label{48}\\
    &\kappa_0=s_{\g}(t_l),~ \kappa_1=s_{\g}(t_l+1),~ \kappa_2=s_{\g}(t_l+2). \label{49}
\end{align}
\end{theorem}
	\begin{proof} We solve the corresponding recurrence equation
		\begin{equation*}
			a(n+3) - (2\cdot b^l + 1)\cdot a(n+2) + (b^{2l}+2\cdot b^l)\cdot a(n+1) - b^{2l}\cdot a(n) = 0,
		\end{equation*}
	and get the basis of solutions
	\begin{equation}\label{50}
		\left\{1,~ b^{l\cdot n},~ n\cdot b^{l\,n}\right\}.
	\end{equation}
	We proceed as in the proof of \Cref{theo1} to find the expected formulae. After solving the linear system of initial conditions, the following is needed to simplify the coefficients $\mu_l$ and $\theta_l$.
	\begin{align}
	&s_{\g}(t_l+2) - s_{\g}(t_l+1) = b^{\nu_l+l}\cdot \left(u(t_l)+2\cdot d\right), \label{51}\\
	&s_{\g}(t_l+1)-s_{\g}(t_l) = b^{\nu_l}\cdot \left(u(t_l)+d\right), \label{52}\\
	&s_{\g}(t_l+2)-s_{\g}(t_l) = b^{\nu_l}\cdot \left(\left(u(t_l)+2\cdot d \right)\cdot b^l + u(t_l)+d\right). \label{53}
	\end{align}
 \CQFD
\end{proof}
Hence, we deduce the formulae for reverse Smarandache numbers.
\begin{corollary}[Formula for reverse Smarandache numbers]\label{cor2} The general term of $(\smr(n))_n$ can be computed as follows:
\begin{align}
&\smr(n) = \alpha_l + \mu_l \cdot 10^{l(n-t_l)} + \theta_l\cdot (n-t_l)\cdot 10^{l(n-t_l)}, \label{54}\\
&l = \lceil \log_{10}\left(n + 2\right)\rceil,~ t_l = 10^{l-1}-1,~  \nu_l=10^{l-1}\cdot\left(l-\frac{10}{9}\right)+l+\frac{1}{9},\label{55}\\
&\alpha_l = \frac{\kappa_2-2\cdot 10^l\cdot \kappa_1 + 10^{2l}\cdot \kappa_0}{(10^l-1)^2}, \mu_l = \frac{10^{\nu_l}\cdot \left(10^{2l-1}-10^{l-1}-1\right)}{\left(10^l-1\right)^2}, \label{56}\\
&\theta_l = \frac{10^{\nu_l}}{10^l-1}, \kappa_0=\smr(t_l),~ \kappa_1=\smr(t_l+1),~ \kappa_2=\smr(t_l+2). \label{57}
\end{align}
\end{corollary}
\begin{proof} Immediate application of \Cref{theo2}. The formula for $\nu_l$ is deduced from the sum
\begin{equation}\label{60}
1+\sum_{k=1}^{l-1}\left(10^k-1-10^{k-1}\right)\cdot k + k+1 = 10^{l-1}\cdot\left(l-\frac{10}{9}\right)+l+\frac{1}{9},
\end{equation}
where $\left(10^k-1-10^{k-1}\right)\cdot k$ counts all the digits of the $k$-digit left-concatenation, and $k+1$ stands for the first $(k+1)$-digit left-concatenation. The extra $1$ before the sum compensates the $1$-digit concatenations as the sequence starts at index $0$. \CQFD
\end{proof}
We end this section with the formula for the palindromic concatenations encoded by $(\sx(n))_n$, which we give without proof.
	
\begin{theorem}\label{theo3} The general term of $(\sx(n))_n,\, b\geq 5$ can be computed as follows:
\begin{align}
&\sx(n) = \alpha_l + \mu_l \cdot b^{l(n-t_l)} + \theta_l\cdot b^{2l(n-t_l)}, \label{61}\\
&l = \lceil \log_{b}(n\,d + \sx(0) + 1)\rceil,~ t_l = \left\lceil\frac{b^{l-1}-\sx(0)}{d}\right\rceil,\, \alpha_l = \frac{b^{3l}\cdot \kappa_0 - b^{l}\cdot \left(b^l + 1\right)\cdot \kappa_1 + \cdot \kappa_2}{\left(b^l+1\right) \cdot \left(b^l-1\right)^2}, \label{63}\\
&\mu_l = - \frac{b^{2l}\cdot \kappa_0 - \left(b^{2l}+1\right)\cdot \kappa_1 + \kappa_2}{b^l\cdot \left(b^l-1\right)^2},\, \theta_l = \frac{b^l\cdot \kappa_0 - \left(b^l+1\right)\cdot \kappa_1 + \kappa_2}{b^l\cdot \left(b^l+1\right) \cdot \left(b^l-1\right)^2}, \label{65}\\
&\kappa_0=\sx(t_l),~ \kappa_1=\sx(t_l+1),~ \kappa_2=\sx(t_l+2). \label{66}
\end{align}
\end{theorem}

\section{Non-holonomicity}\label{sec4}

Flajolet, Gerhold, and Salvy \cite{flajolet2005non} proposed an important strategy to prove that a sequence is not holonomic. Their method arises from the behavior of the corresponding generating functions around their singularities. Indeed, the possible growth of a D-finite function $f(x)\coloneqq \sum_{n=n_0}^{\infty} a(n)\,x^n$ (see \cite{stanley1980differentiably}) is highly constrained near any of its singularities. By \textit{Abelian theorems} (see \cite{Regvar}, \cite[Theorem 3]{flajolet2005non}), these constraints are transferred to the asymptotic behavior of the power series coefficients $a(n)$ at infinity, excluding thus all sequences that do not satisfy these constraints. As a result, Flajolet, Gerhold, and Salvy wrote, `{\it almost anything is not holonomic unless it is holonomic by design}'. Another asymptotic approach is given in \cite{bell2008non} for sequences $a(n)=\left.f(x)\right|_{x=n}$, where $f$ is an explicitly known function. We mention that Gerhold initially proved that fractional powers of hypergeometric sequences are not holonomic by using facts about algebraic extensions \cite{gerhold2004some}. In this section, we also provide proofs not directly related to the singularity analysis of the generating functions. Our first proof relies on basic facts from difference algebra in the principal ideal domain of univariate shift operators. The second proof is deduced from arithmetic relations in these sequences at large indices.

\subsection{Proof from shift algebra}\label{subsec:mainproof}

Let $\KK \supset \QQ$ be a field of characteristic zero. We consider the ring of linear operators $R_{\sigma}\coloneqq \KK(n)\langle \sigma \rangle$, where $\sigma$, denoting the shift operator, acts in the following manner
\[\sigma\cdot f(n) = f(n+1)\sigma,\, \forall\, f\in \overline{\KK(n)}.\]
The field $\KK(n)$ is seen as the difference field $(\KK(n),\sigma)$, and $\overline{\KK(n)}$ denotes its closure. Any operator $p\in\rs$ has the form
\[p=p_0+p_1\sigma+\cdots+p_r\sigma^r,\]
where $r=\ord(p)$ is the order of $p$. The action $p\cdot a(n)$, of the operator $p$ on a sequence general term $a(n)$ is the linear combination
\[p_0(n)\,a(n)+p_1(n)\,a(n+1)+\cdots+p_r(n)\,a(n+r).\]
In this setting, we say that a sequence $(a(n))_{n\in\NN}$ is holonomic if there exists $p\in \rs$ such that
\begin{equation}\label{eq:pan0}
    p\cdot a(n) = 0,\, \forall n\in\NN;
\end{equation}
in which case, $p$ is called an annihilator of $(a(n))_{n\in\NN}$. The sequence $(a(n))_{n\in\NN}$ is thus uniquely defined with $p$ and $\max(r,N+1)$ initial values, where $N$ is the maximum integer root of the leading polynomial coefficient $p_r$. We often identify holonomic sequences with an annihilating operator and their initial values. However, to simplify the text, we usually omit mentioning the initial values when considering operators for sequences.

We also write $(a(n))_{n\in\NN}\in \sol(p)$, where $\sol(p)$ is the set of solutions to the $P$-recursive equation encoded by $p$. We denote the corresponding vector space by $\langle \sol(p) \rangle$. For further details about univariate linear difference operators,  see, for instance, \cite{bronstein1996introduction,ore1933theory}. We will use the following well-known facts:
\begin{enumerate}
    \item $\forall p\in \rs,$ $\dim\left(\langle \sol(p) \rangle\right)<\infty$.
    \item $\forall p,q\in \rs$, the ideal generated by $p$ and $q$ is such that $\langle p, q\rangle = \langle \gcd(p,q)\rangle,$ where $\gcd$ denotes the \textit{greatest common right divisor} of $p$ and $q$. Thus, every ideal is principal. Moreover $\sol(\gcd(p,q))=\sol(p)\cap \sol(q).$ 
\end{enumerate}

Note that for any $p\in\rs$, the encoded equation is equivalent to the one obtained after clearing the denominators. For this reason, we always assume that our operators have polynomial coefficients. It follows from the above facts that the ideal generated by the annihilators of a holonomic sequence is generated by a single operator whose order is minimal. 

We use the following definition as an essential tool to regard holonomic equations and their initial values as closed forms without necessarily constructing formulae. 

\begin{definition}[Germ of a holonomic sequence (see Example 1.3 in \cite{vansinger1997galois})] The germ of a holonomic sequence $(a(n))_{n\in\NN}$ or an operator $p\in\rs$ that annihilates it is the sequence $(a(n+N))_{n\in\NN}$ or the operator $\sigma^N \cdot p$, where $N$ is an integer greater than all integers where the polynomial coefficients in $p$ have some observed local behavior.
\end{definition}

Two holonomic sequences have the same germ if their difference has finite support \cite{van1999finite}. For operators, this means that one is a shift of the other. Notice that the shifting by $N$ is equivalent to keeping the same operator or sequence, but starting from index $N$.

Our interest in considering germs of holonomic sequences relies on the fact that they represent the generic solutions of holonomic equations. In the following example, we highlight a difference between what we may call {\it holonomic by closed forms} and {\it holonomic by values}, which is central to our reasoning on integer intervals. This serves as a preview for the coming definition, which helps formalise the concept of being holonomic per range and understand what the proof of the guessed equations represents.

\begin{example}[Holonomic by values VS holonomic by closed form (or design)]\label{eg:hvVhc} Let $p$ as in \eqref{eq:pan0}, $I\subset \NN$ a range, and suppose the sequence $(a(n))_{n\in\NN}$ satisfies the following relation ``by design''.
\begin{equation}\label{eq:egeq1}
    p\cdot a(n)=0,\,\, \text{for all }n\in I.
\end{equation}
What we mean by design is that $(a(n))_{n\in\NN}$ satisfies \eqref{eq:egeq1} as the germ of the operator $p$ does, meaning that the relation is not caused by local properties of $p$ in $I$ but the behavior of $(a(n))_{n\in\NN}$. This is what we understand as closed form in the interval $I$. It does not have to be necessarily an explicit formula.

Let now $p_{r+1}$ be a polynomial such that $p_{r+1}(n)=0$, for all $n\in I$. Then the operator $q=~p_{r+1}\sigma^{r+1}+p$ also ``cancels'' $(a(n))_{n\in\NN}$ on $I.$ This is what we regard as ``artificial holonomicity'' or holonomicity by values, because the polynomial $p_{r+1}$ is not constructed from a ``genuine'' behavior of $(a(n))_{n\in \NN}$. The key difference is that the germs of $p$ and $q$ are completely different. One can construct artificial operators like $q$ for any non-polynomial holonomic sequence by using polynomial interpolation on the given range. This would yield a first-order operator where the coefficients are the interpolating polynomial and its first shift.
\end{example}

To be more formal and make this distinction more precise, we introduce a new type of holonomic representation that generalises the classical one.

\begin{definition}[Holonomicity per range (or integer interval)]\label{def:holperrange} Let $(a(n))_{n\in\NN}$ be a sequence, and $(I_j)_{j\in S},$  be a partition of $\NN$ in ranges such that $S\coloneqq \{0,1,2,\ldots,M\}, M\in\NN$, $\max\{I_j\}+1=\min\{I_{j+1}\}$, with $\min\{I_0\}=0.$ We say that $(a(n))_{n\in\NN}$ is holonomic per range if there exists a family of holonomic sequences $\left((a^{(j)}(n))_{n\in\NN}\right)_{j\in S}$ such that
\begin{equation}\label{eq:holperrange}
 a(n)=a^{(j)}(n),\,\, \forall n\in I_j,\,\, \text{and}
\end{equation}
the germ of $(a^{(j)}(n))_{n\in\NN}$ satisfies the same holonomic equation that $(a(n))_{n\in\NN}$ satisfies in $I_j$.
\end{definition}
We denote a partition with the properties of $(I_j)_{j\in S}$ in \Cref{def:holperrange}, a {\it natural partition of $\NN$} or simply a natural partition.

The statement ``the germ of $(a^{(j)}(n))_{n\in\NN}$ satisfies the same holonomic equation that $(a(n))_{n\in\NN}$ satisfies in $I_j$'' is equivalent to ``in $I_j$, $a(n)$ has the same general formula that the germ of $a^{(j)}(n)$ has for large integers.'' In other words, the identity $a(n)=a^{(j)}(n)$ is not restricted to local properties of the annihilator of $(a^{(j)}(n))_{n\in\NN}$ in $I_j$. This excludes all artificial operators that are over-fitted to the ranges.

\begin{example} \item 
\begin{itemize}
    \item The sequences $(s(n))_{n}, (s_{\g}(n))_n,$ and $(s_*(n))_n$ are holonomic per range. This is a direct consequence of our proof that the guessed recurrence equations are satisfied by the design of these sequences.
    \item Every holonomic sequence $(a(n))_{n\in\NN}$ is holonomic per range. A direct way is to use the definition of $(a(n))_{n\in \NN}$ with its annihilator and initial values, and the natural partition as the trivial partition that only contains $\NN$. The converse is not true, and we are going to demonstrate it with $(s(n))_{n}, (s_{\g}(n))_n,$ and $(s_*(n))_n$.
\end{itemize}
\end{example}

Our next definition introduces a more interesting subclass of sequences that are holonomic per range. 

\begin{definition}[Global holonomic sequence] Let $(a(n))_{n\in\NN}$ be holonomic per range with natural partition $\mathcal{I}(a)=(I_j)_{j\in S}$ and holonomic sequences $\mathcal{A}(a)\coloneqq \left((a^{(j)}(n))_{n\in\NN}\right)_{j\in S}$ as previously. We say that $(a(n))_{n\in\NN}$ is globally holonomic with $\mathcal{I}(a)$ if there exists an operator $p$ of order $r$ such that the sequences defined by $p$ and the initial values $a^{(j)}(i)$, $i=0,\ldots,$ $\max(r-1,N)$ are exactly the family $\mathcal{A}(a)$. Here $N$ is the maximum integer root of the leading polynomial coefficient of $p$.
\end{definition}

\begin{example} Consider the sequence $(u(n))_{n\in\NN}$ defined by 
    \[ u(n) \coloneqq \begin{cases} n\, \text{ if } n \leq 11\\ n!\, \text{ otherwise}\end{cases}.\]
    $(u(n))_{n\in\NN}$ is clearly holonomic per range since we have $p\cdot u(n)=0$ for $n\in \llbracket 0, 10 \rrbracket$ due to its formula, where $p=n\sigma - (n+1)$. We also have $q\cdot u(n)=0$ for $n\geq 12$, with $q=\sigma - (n+1)$. Constructing an operator that annihilates the solutions to $p$ and the solutions to $q$ is straightforward. Here is an example:
    \[\rho\coloneqq \left(n^{2}-1\right)\sigma^2 -\left(n+2\right) \left(n^{2}+n-1\right)\sigma +  n \left(n+2\right) \left(n+1\right).\]
    One verifies that the sequences $(n)_{n\in\NN}$ and $(n!)_{n\in\NN}$ can both be defined using $\rho$ and their first three initial values. Therefore $(u(n))_{n\in\NN}$ is globally holonomic.
\end{example}

The proof of the following proposition is straightforward, but it also gives us a necessary condition for a sequence holonomic per range to be holonomic.

\begin{proposition} Every holonomic sequence is globally holonomic.    
\end{proposition}

It might be interesting to look at the converse of this proposition. Nevertheless, this is not needed for our purpose. We have the following diagram.

\[\text{Holonomic} \quad \subset  \quad \text{Globally holonomic} \quad \subset \quad \text{Holonomic per range.}\]

In \Cref{sec3}, we have seen that $s(n)=\alpha_l+\mu_l\,(n-t_l)+\theta_l\,b^{l(n-t_l)}$. Since $\alpha_l+\mu_l\,(n-t_l)$ (see \eqref{eq:eq12})  is a rational (actually polynomial) function in $n$, it can be seen as a single hypergeometric term. Thus, $l$-digit concatenations in $(s(n))_n$ satisfy a second-order holonomic RE. Similarly, $s_{\g}(n) =\alpha_l + \left(\mu_l + \theta_l\,(n-t_l)\right)\, b^{l(n-t_l)}$ (see \eqref{18}) can be seen as a linear combination of two linearly independent hypergeometric terms, and therefore its recurrence equation reduces to a second-order holonomic RE. For $(\sx(n))_n$, no reduction is possible (see \eqref{24}) because its characteristic polynomial has three distinct roots. We mention that the corresponding equations can be computed with the algorithms from \cite{BThyper} and \cite[Section 3.2]{teguia2023hypergeometric}. 

We are now ready to state and prove the main theorem of this section.

\begin{theorem}\label{th:theo4} The sequences $(s(n))_n, (s_{\g}(n))_n$, and $(\sx(n))_n$ are not holonomic.
\end{theorem}
\begin{proof} We prove by contradiction that these sequences are not globally holonomic, thereby proving their non-holonomicity. Let $(a(n))_n$ be any of the sequences $(s(n))_n, (s_{\g}(n))_n$, and $(\sx(n))_n$, and suppose that $(a(n))_n$ is globally holonomic. Let $p\in\rs$ be the minimal operator of order $r$ that annihilates $(a(n))_n$. We can assume that $r\geq 2$ for $(s(n))_n$ and $(s_{\g}(n))_n$, and $r\geq 3$ for $(\sx(n))_n$; the proof will show this is the only choice. Let $l\geq r$ and consider $l$-digit concatenations encoded by the minimal annihilating operator $q_l$. From the previous paragraph, we know that $\ord(q_l)=2$ for $(s(n))_n$ and $(s_{\g}(n))_n$, and $\ord(q_l)=3$ for $(\sx(n))_n$. We combine these two cases by writing $\ord(q_l)=m$.

We can find $l\geq m$ such that $\ord(\gcd(p,q_l))<m$. Indeed, from the explicit formulae in \Cref{sec3}, it follows that for any $i$-digit concatenations and $j$-digit concatenations, $i\neq j$, the respective encoding operators $q_i$ and $q_j$, are such that
\begin{equation}\label{eq:condsol}
\dim\left(\langle \sol(q_i) \rangle \bigoplus \langle \sol(q_j)\rangle\right) > \dim\left(\langle \sol(q_i)\right). 
\end{equation}
This implies that 
\begin{equation}
    \dim(\bigoplus_{i\geq l} \langle \sol(q_i) \rangle) = \infty,
\end{equation}
and therefore cannot be spanned by $\sol(p)$. Thus, there exists $l\geq r$ such that $q_l$ is not a right-divisor of $p$, and so $\ord(\gcd(p,q_l))<m$. However, for such an $l$, for all integer $n\in\llbracket t_l,\,t_{l+1}-1 \rrbracket$, where $t_l$ is the index of the first $l$-digit concatenation, we have
\[q_l\cdot a(n) =0, \,\, \textnormal{and} \,\, p\cdot a(n) =0.\]
We remind that the meaning here is that the solutions of $q_l$ can be expressed as a linear combination with the basis of solutions of $p$. So we must also have $\gcd(q_l,p)\cdot a(n) =0$ for all $n\in\llbracket t_l,\,t_{l+1}-1 \rrbracket$ as the sequence is supposed to be globally holonomic. In other words, the formulae we obtained for $l$-digit concatenations can be deduced from the solutions of lower-order recurrence equations. This is absurd because $q_l$ is minimal. Hence $(s(n))_n$, $(s_{\g}(n))_n$ and $(\sx(n))_n$ are not globally holonomic, and therefore not holonomic. \CQFD
\end{proof}

The non-holonomic character of these sequences can be seen as a consequence of the fact that \eqref{eq:condsol} holds. This condition presents a way to construct non-holonomic sequences. Indeed, since the solutions of the recurrence equations in \Cref{sec2} are obtained with the $m$-fold hypergeometric solver from \cite{teguia2021symbolic}, it follows that concatenations of terms of an increasing integer-valued hypergeometric-type sequence (see \cite{teguia2023hypergeometric}) produce non-holonomic sequences. For instance, concatenations of terms of the sequence of general term $2^n+\chi_{\{n\equiv 1\mmod 2\}}$ yield a non-holonomic sequence. For observation purposes, one can verify that the guessing algorithms from the \texttt{GFUN} package, \cite{kauers2015ore,kauers2022guessing}, or \cite{BTguessing} return no differential equation for the corresponding generating function. On the other hand, we can relate the non-holonomic character of these sequences to the apparent non-existence of their generating functions as a differentiable object. Indeed, using Pad\'{e} approximants, we can accurately approximate their generating functions at every range of concatenations. However, given the `brutal' changes at the endpoints of these ranges, it occurs that if $F$ is the generating function of one of these sequences, then 
\begin{equation}
    (F^{(b^n-2)})'\neq F^{(b^n-1)}, \forall\, n\in\NN\setminus \{0\}.
\end{equation}
This precludes $F$ from being a differentiable function and, therefore, cannot be D-finite.

\subsection{Alternative Proof}

We propose an alternative proof of \Cref{th:theo4}. This is self-contained and independent of the recurrence equations from \Cref{sec2}. We give full details for the case of $(s(n))_n$ and adapt the reasoning for $(s_{\g}(n))_n$. A similar reasoning applies to the case of $(\sx(n))_n$.

\subsubsection{Case of $(s(n))_n$}\label{subsec:sn}

\begin{proof}
As previously, we denote by $t_k$ the index of the first $k$-digit concatenation in $(s(n))_n$. Without loss of generality, we assume that $d<b$. It is immediate to see that $t_k=b^{O(k-1)}$. This is just an approximation of how big $t_k$ is compared to $k$. Suppose that $k$ is very large. We look at the relations between terms of $s(n)$ for $n\in\llbracket t_k, t_{k+1}-1 \rrbracket$. 

Let $A_k=s(t_k)$. Since there are $t_{j+1} -  t_j$ terms of $(u(n))_n$ with $j$ digits, $j\in\NN$, it follows that 
\begin{equation}\label{eq:dblexp}
    A_k \gg b^{\sum_{j=1}^{k-1} (t_{j+1}-t_j)} = b^{O(b^{k-1})},
\end{equation}
a double exponential quantity in $k$. Now suppose that $(s(n))_n$ is holonomic such that 
\begin{equation}\label{eq:hypseq}
    \sum_{j=0}^r s(n+j)\, p_j(n) = 0,
\end{equation}
where the polynomials $p_j(n)\in\ZZ[n]$ are not all zeros, and $r<t_{k+1}$. Observe that
\begin{align*}
  s(t_k) &= A_k,\\
  s(t_k+1) &= A_k\,b^k + u(t_k)+d = A_k\,b^k + b^{O(k)},\\
  s(t_k+2) &= A_k\,b^{2k} + (u(t_k)+d)\,b^k+u(t_k)+2\,d=A_k\,b^{2k} + b^{O(k)}.
\end{align*}
Thus, for all integers $i,j \leq r$
\begin{equation}\label{eq:Ostk}
    s(t_k+i+j) = A_k\,b^{(i+j)k}+b^{O(k)}.
\end{equation}
We plug \eqref{eq:Ostk} into \eqref{eq:hypseq} to get
\[\sum_{j=0}^r s(t_k+i+j) p_j(t_k+i) = \sum_{j=0}^r (A_k\,b^{(i+j)k}+b^{O(k)})p_j(t_k+i) = 0.\]
Note that in the summation above, $b^{O(k)}$ implicitly depends on the summation index $j$. Its explicit form is irrelevant to the arguments of the proof.

Since $p_j(n) = n^{O(1)}$ and $t_k=b^{O(k)}$, we can rewrite the above equation as follows:
\begin{equation}\label{eq:skeyeq}
    A_k \sum_{j=0}^r b^{(i+j)k} p_j(t_k+i) = - \sum_{i=0}^r b^{O(k)} p_j(t_k+i) = b^{O(k)}.
\end{equation}
Thus, the double exponential quantity $A_k$ divides the right-hand side above, which is only single exponential in $k$. Hence, the right-hand side must be $0$, and thus, 
\begin{equation}\label{eq:freeAk}
    \sum_{j=0}^r b^{(i+j)k} p_j(t_k+i) = 0.
\end{equation}
Note that all recurrence equations from \Cref{sec2} satisfy such a relation.

Let $X=t_k+i.$ Thus, $t_k=X-i$. We know that $t_k=\left\lceil \frac{b^{k-1}-s(0)}{d} \right\rceil$. We have
\[(d\,t_k+s(0)) -  d < b^{k-1} \leq d\,t_k+s(0).\]
So, $b^{k-1}=d\,t_k+s(0)+\Delta$, with a fixed $\Delta\in\llbracket -d+1,\, 0\rrbracket$. Let
\begin{equation}\label{eq:fxi}
    f(X-i)\coloneqq b\left(d\,(X-i)+s(0)+\Delta\right) = b^k.
\end{equation}
Then \eqref{eq:freeAk} becomes
\begin{equation}\label{eq:eqfxj}
    \sum_{j=0}^r f(X-i)^{i+j}\,p_j(X)=0.
\end{equation}
Keeping $i$ fixed, the left-hand side is a polynomial in $X$ with infinitely many zeros, namely all $X=t_k+i$ for large $k$'s. Thus, \eqref{eq:eqfxj} is identically $0$ for all $X$. Simplifying $f(X-i)^i$ in \eqref{eq:eqfxj} we get the system
\begin{equation}
    \sum_{j=0}^r f(X-i)^j\,p_j(X)=0, \quad\, i=0,\ldots,r.
\end{equation}
This means that the vector of polynomials $(p_0(X),p_1(X),\ldots,p_r(X))^T$ is orthogonal to the following Vandermonde matrix
\[ V_f\coloneqq
    \begin{bmatrix}
        1      & f(X)    &   f(X)^2    & \cdots &  f(X)^r\\
        1      & f(X-1)  &   f(X-1)^2  & \cdots &  f(X-1)^r\\
        1      & f(X-2)  &   f(X-2)^2  & \cdots &  f(X-2)^r\\
        \vdots & \vdots  &   \vdots    & \ddots &  \vdots\\
        1      & f(X-r)  &   f(X-r)^2  & \cdots &  f(X-r)^r
    \end{bmatrix}.\]
Since all $f(X-i)$'s are distinct, $V_f$ is invertible. This implies that $(p_0(X),p_1(X),\ldots,p_r(X))^T$ is the zero vector, leading thus to the desired contradiction.\CQFD
\end{proof}

This proof may be applied to every sequence of general term $\overline{P(0)P(1)\cdots P(n)}^b$, where $P$ is some positive polynomial over $\ZZ$. The point is that a relation of the form of \eqref{eq:Ostk} can be obtained modulo $b^k$. Such a non-holonomicity is expected since the discussion at the end of \Cref{subsec:mainproof} also covers this case.

\subsubsection{Case of $(s_{\g}(n))_n$}\label{subsec:sr}

\begin{proof}
   We wish to obtain an equation in the form of \eqref{eq:freeAk}.
   
   Let $A_k=s_{\g}(t_k)$, with $t_k$ defined as previously. Let $\nu_k$ be the digit length of $A_k$. By analogy to the computation in \eqref{60}, one establishes that $\nu_k=kb^{O(k)}$. This implies that $A_k\gg b^{\nu_k} = b^{O(kb^k)}$. For all integers $i,j\leq r<t_{k+1}$, we have
   \begin{eqnarray}\label{eq:Osrtk}
        s_{\g}(t_k+i+j) &=& A_k + b^{\nu_k}\left(\sum_{l=1}^{i+j} ( u\! \left(t_k\right)+l\,d)b^{(l-1)\,k}\right)\nonumber\\
                     &=& A_k + b^{\nu_k}\, C(i,j,k), \label{eq:srijk}
    \end{eqnarray}
where 
\begin{equation}\label{eq:cij}
    C(i,j,k)=\frac{\left(\left(d\left(i+j\right)+u\! \left(t_{k}\right)\right) b^{k}-u\! \left(t_{k}\right)-d\left(1+i+j\right)\right) \left(b^{k}\right)^{i+j}-u\! \left(t_{k}\right) b^{k}+u\! \left(t_{k}\right)+d}{\left(b^{k}-1\right)^{2}}.
\end{equation}
The holonomic equation for $s_{\g}(t_k+i+j)$ is then equivalent to
\begin{equation}
    A_k \sum_{j=0}^r p_j(t_k+i) = - b^{\nu_k}\,\sum_{j=0}^r C(i,j,k)\,p_j(t_k+i).
\end{equation}
One may assume that $A_k$ and $b^{\nu_k}$ have a bounded $\gcd$. This is clearly the case when $u(0)$ and $b$ are coprime, for example. For the general case, note that we can always ignore any number of starting terms from our concatenations. For example, since the sequence of general term $s_{\g}(n)$ is holonomic, so is the sequence of general term $(s_{\g}(n)-u(0))/b^{\ell_0}$, where $\ell_0$ is the number of base $b$ digits of $u(0)$. This is nothing else but the sequence whose general term is the left--concatenation of $u(1),~u(2),~\ldots$ (with $u(0)$ removed). Proceeding in this way, we may assume that we remove the first $j$ terms and start with $u(j)=u(0)+dj$, where $j$ is large enough such that the following two conditions are satisfied:
\begin{itemize}
\item[(i)] Putting $D:=\gcd(u(0),d)$, we have $u(j)=D\rho$, where $\rho>b$ is prime.
\item[(ii)] Putting $\ell_j$ for the number of digits of $u(j)$ in base $b$, we have $\ell_j>\max\{\mu_{\rho}(D): \rho\mid b\}$, where $\ell_j$ is the number of digits of $u(j)$ in base $b$.
Here, $\mu_{\rho}(D)$ is the exponent of $\rho$ in the factorization of $D$. 
\end{itemize}
The existence of $j$ above follows from Dirichlet's theorem on prime in progressions. Then with these choices, one can see that $\gcd(A_k,b^{\nu(k)})\le D$ for all $k\ge 1$. 
Let $A_k'=A_k/\gcd(A_k,b^{\nu_k})$, and $b_{\nu_k}'=b^{\nu_k}/\gcd(A_k,b^{\nu_k})$. Thus, $A_k'$ divides the sum on the right-hand side, and $b_{\nu_k}'$ divides the one on the left-hand side. However, 
\begin{equation}
    \sum_{j=0}^r p_j(t_k+i) = b^{O(k)}, \quad\, \sum_{j=0}^r C(i,j,k)\,p_j(t_k+i)= b^{O(k)},
\end{equation}
while $A_k'$ and $b_{\nu_k}'$ are both double exponential. Therefore, we must have
\begin{align}
    &\sum_{j=0}^r p_j(t_k+i)=0, \label{eq:scoef}\\
    &\sum_{j=0}^r C(i,j,k)\,p_j(t_k+i)=0. \label{eq:rhssr}
\end{align}
Remark that \eqref{eq:scoef} is verified by all equations in \Cref{sec2}: the sum of the coefficients is zero. Using \eqref{eq:scoef} and \eqref{eq:cij}, we simplify \eqref{eq:rhssr} and get the equation
\begin{equation}
    \sum_{j=0}^r \left(u\! \left(t_{k}+i+j\right)b^k-u\! \left(t_k+i+j+1\right)\right) b^{k(i+j)} p_j(t_k+i) = 0.
\end{equation}
Following the same reasoning as previously, we arrive at the following system of equations
\begin{equation}
    \sum_{j=0}^r g\left(f(X-i),i+j\right) f(X-i)^j p_j(X) = 0,\quad\, i=0,\ldots,r,
\end{equation}
where $f$ keeps a similar definition as before, and $g$ is defined appropriately, noting that $u(t_k) = s_{\g}(0)+d\,t_k= f(X-i)/b-\Delta$. One verifies that the resulting matrix is full-rank and obtains the desired contradiction.
\CQFD 
\end{proof}

\section{Computations}\label{sec5}
	
 By Theorems \ref{theo1}, \ref{theo2}, and \ref{theo3}, one can write algorithms for efficient computation of terms of our concatenating sequences. We implemented the particular cases of Smarandache numbers $(\sm(n))_{n\in\NN}$ and their reverse $(\smr(n))_{n\in\NN}$ in Maple. The resulting code and software are available from the link \url{https://github.com/T3gu1a/Concatenations}. 
	
\begin{algorithm}[!ht]
    \caption{$\sm(n)$}\label{Algo1}
    \begin{algorithmic}
        \Require  A non-negative integer $n$.
        \Ensure $\sm(n)$: the $(n+1)^{\text{st}}$ Smarandache number.
        \begin{enumerate}
            \item if $\sm(n)$ is defined then stop and return $\sm(n)$
            \item $l=\lceil\log_{10}(n+2)\rceil$	
            \item if $\alpha_l$ is not defined then
                \begin{enumerate}
                    \item $d=(10^l-1)^2$
                    \item $t_l=10^{l-1}-1$ and save $t_l$
                    \item $\alpha_l=-\left(10^{2l-1}+9\cdot10^{l-1}\right)/d$ and save $\alpha_l$
                    \item $\mu_l=-1/\left(10^l-1\right)$ and save $\mu_l$
                    \item $\theta_l = \texttt{numtheaSm(l)}/d$ and save $\theta_l$
                \end{enumerate} 
            \item Return and save $\sm(n)=\alpha_l+\mu_l\cdot\left(n-t_l\right) +\theta_l\cdot10^{l(n-t_l)}$
        \end{enumerate}
    \end{algorithmic}
\end{algorithm} 

The numerator of $\theta_l$ is computed by \Cref{Algo2}.

\begin{algorithm}[!ht]
    \caption{\texttt{numthetaSm(l)}}\label{Algo2}
    \begin{algorithmic}
        \Require  A positive integer $l$.
        \Ensure The numerator of $\theta_l$ in Algorithm \ref{Algo1}.
        \begin{enumerate}
            \item if $l=1$ then return $100$
            \item $s_0=\texttt{conc}_l(\sm(t_l-1),t_l+1)$
            \item $s_1=\texttt{conc}_l(s_0,t_l+2)$
            \item $s_2=\texttt{conc}_l(s_1,t_l+3)$
            \item return $s_2-2\cdot s_1 + s_0$
        \end{enumerate}
    \end{algorithmic}
\end{algorithm} 

The concatenation $\texttt{conc}_l$ is defined as: $(a,b) \mapsto a\cdot 10^{l}+b$. \Cref{Algo1} uses a remembering effect for $t_l$ $\alpha_l$, $\mu_l$, $\theta_l$, and the returned values. This helps to avoid computing the same values several times and is especially needed for the coefficients $\alpha_l$, $\mu_l$, and $\theta_l$, which are used several times. For $\smr(n)$, the algorithm is similar to \Cref{Algo1}. The formulae are the only changes, and $\nu_l$ is also computed with a remembering effect.

We now compare the efficiency of the resulting implementation with the other existing codes for $\sm(n)$ and $\smr(n)$ in Maple (see the codes from \cite{torres2004smarandache},  \href{http://oeis.org/A007908}{oeis A007908}, and \href{https://oeis.org/A000422}{oeis A000422}). We use Maple 2023. The most efficient for $\sm(n)$ seems to be the following.
\begin{lstlisting}
> sm := n-> parse(cat(`$`(n+1))):
\end{lstlisting}
For reverse Smarandache numbers, we define \texttt{smr} similarly.
	
\begin{lstlisting}
> smr := proc(n, $) local i; parse(cat((n + 1 - i) $ (i = 0..n))) end proc:
\end{lstlisting}
We do not consider recursive implementations because we will evaluate at very distant indices. The recursive approach is more suitable for printing out the `triangle of the gods': printing consecutive terms of $(\sm(n))_{n\in\NN}$, one per line. We use the Maple command \texttt{CPUtime} from the package \texttt{CodeTools} to display the CPU times. The computations are summarised below in \Cref{tab1} and \Cref{tab2}.

\begin{table}[!ht]
    \begin{center}
    \caption{\texttt{Smarandache:-Sm} vs \texttt{sm}}
    \label{tab1}
    \begin{tabular}{|l|l|l|l|l|}
    \hline
    $l$                                                           & $5$    & $6$      & $7$          & $8$       \\ \hline
    \texttt{CPUTime}(\texttt{Smarandache:-Sm}($10^l-1$))         & $0.046$ & $0.125$  & $1.766$      & $31.532$     \\ \hline
    \texttt{CPUTime}(\texttt{sm}($10^l-1$))                      & $0.079$ & $0.719$  & $10.969$     & $208.391$ \\ \hline
    \end{tabular}
\end{center}
\end{table}
 
One observes that \texttt{Smarandache:-Sm} is faster than \texttt{sm} for asymptotic computations. However, these two codes can be combined to compute closer terms using \texttt{sm} and distant terms using \texttt{Smarandache:-Sm}.

\begin{table}[!ht]
    \begin{center}
    \caption{\texttt{Smarandache:-Smr} vs \texttt{smr}}
    \label{tab2}
    \begin{tabular}{|l|l|l|l|l|}
        \hline
        $l$                                                        & $5$      & $6$     & $7$      & $8$           \\ \hline
        \texttt{CPUTime}(\texttt{Smarandache:-Smr}($10^l-1$))      & $0.016$  & $0.516$ & $7.313$  & $123.657$     \\ \hline
        \texttt{CPUTime}(\texttt{smr}($10^l-1$))                   & $0.047$  & $1.047$ & $12.921$ & $215.765$     \\ \hline
    \end{tabular}
\end{center}
\end{table}

Reverse Smarandache numbers are more difficult to compute with mathematical formulae. The presence of $\nu_l$ explains this (see Corollary \ref{cor2} and Theorem \ref{theo2}). It also explains why the coefficients $\alpha_l$, $\mu_l$, and $\theta_l$ have more decimal digits in left-concatenations. And again, the mathematical formulae yield a more efficient implementation.
	
\section{Conclusion}\label{sec6}
	
The main results of this article are given by \Cref{theo1}, \Cref{theo2}, \Cref{theo3}, and \Cref{th:theo4}. The latter is particularly interesting to the community in difference algebra, as our given proof reveals a strategy for generating non-holonomic sequences. The obtained formulae for concatenations of terms of an arithmetic progression enabled us to exhibit algorithms for computing terms of the sequences \href{http://oeis.org/A007908}{oeis A007908}, \href{https://oeis.org/A000422}{oeis A000422}, and their term-wise concatenation \href{https://oeis.org/A173426}{oeis A173426}. Another sequence in this context is the concatenation of odd integers, a particular case of \Cref{theo1} in the decimal base with the common difference $d=2$. We implemented the formulae for Smarandache numbers and their reverses (\Cref{cor1} and \Cref{cor2}) in Maple. The resulting software is available on GitHub at \url{https://github.com/T3gu1a/Concatenations}.

\medskip

\textbf{Acknowledgment.} We thank the anonymous referees for their helpful comments and stimulating discussions, which prompted us to introduce Definitions 2 and 3. We thank Andrew Scoones and James Worrell for numerous helpful discussions. Both authors were partly supported by UKRI Frontier Research Grant EP/X033813/1.

\end{document}